\documentclass[a4paper,11pt,reqno]{amsart}

\usepackage{amsmath,amssymb,amsthm}
\usepackage{enumerate}
\usepackage[margin=1in]{geometry}
\usepackage[colorlinks=true,urlcolor=black,citecolor=black,linkcolor=black]{hyperref}

\newcommand{\C}{\mathbb C}
\newcommand{\N}{\mathbb N}
\newcommand{\Q}{\mathbb Q}
\newcommand{\D}{\mathbb D}
\newcommand{\Pp}{\mathbb P}
\newcommand{\E}{\mathbb E}
\newcommand{\ind}{\mathbf 1}
\newcommand{\Real}{\operatorname{Re}}
\DeclareMathOperator{\Av}{Av}

\theoremstyle{plain}
\newtheorem{theorem}{Theorem}[section]
\newtheorem{lemma}[theorem]{Lemma}
\newtheorem{corollary}[theorem]{Corollary}

\title[Cofinite Zeros of High Derivatives]{Cofinite Zeros of High Derivatives}
\author{Eric Hou}
\address{Independent researcher, Los Gatos, California, USA}
\email{eric.x.hou@gmail.com}


\begin{document}

\begin{abstract}
We give a bounded-coefficient probabilistic construction of a transcendental
entire function $f$ of order two such that every nonempty open subset of the
complex plane contains a zero of $f^{(n)}$ for all sufficiently large $n$.
This gives an affirmative answer to the transcendental form of Erd\H{o}s
Problem~906.
Thus every fixed disk is zero-free for only finitely many successive
derivatives. The function satisfies $|f(z)|\leq\sqrt2\exp(|z|^2)$ and is a
counterexample to a theorem of Boas and Reddy as printed.
\end{abstract}

\maketitle

\section{Introduction}

Erd\H{o}s asked whether there is a nonzero entire function $f$ such that, for
every strictly increasing sequence $(n_k)_{k\geq0}$ of nonnegative integers,
the union of the zero sets of $f^{(n_k)}$ is dense in the complex plane
\cite[p.~72]{Erdos1982}. On the same page he wrote that the existence of both
functions in his question had been proved ``more than ten years ago''; the
attached note lists four papers of Barth and Schneider, including
\cite{BarthSchneider1972}. The question is now recorded as Problem~906
\cite{Bloom906}. If polynomials are admitted it is immediate, so we treat the
transcendental form. Write $\N=\{0,1,2,\ldots\}$ and
$B(c,R)=\{z\in\C:|z-c|<R\}$.

For a nonempty open set $U\subset\C$ and an entire function $f$, define
\[
 A_U(f)=\{n\in\N:f^{(n)}\text{ has a zero in }U\}.
\]
The following elementary reformulation isolates the quantifier that the
problem requires.

\begin{lemma}\label{lem:cofinite}
For an entire function $f$, the following assertions are equivalent.
\begin{enumerate}[\rm(1)]
\item For every strictly increasing sequence $(n_k)_{k\geq0}$ in $\N$, the set
\[
 \bigcup_{k\geq0}\{z\in\C:f^{(n_k)}(z)=0\}
\]
is dense in $\C$.
\item For every nonempty open set $U\subset\C$, the set $A_U(f)$ is cofinite
in $\N$, that is, $\N\setminus A_U(f)$ is finite.
\end{enumerate}
\end{lemma}

\begin{proof}
Assume (2). Fix a strictly increasing sequence $(n_k)$ and a nonempty open
set $U$. Since $A_U(f)$ is cofinite there is an integer $N(U)$ such that
$n\geq N(U)$ implies $n\in A_U(f)$. The sequence $(n_k)$ is unbounded, so
$n_k\geq N(U)$ for some $k$, and then $f^{(n_k)}$ has a zero in $U$. The
displayed union therefore meets every nonempty open set and is dense.

Conversely, suppose that (2) fails. Then for some nonempty open set $U$ the
set $\N\setminus A_U(f)$ is infinite; enumerate it increasingly as $(n_k)$.
Each $f^{(n_k)}$ is zero-free in $U$, so the union of their zero sets misses
$U$ and is not dense.
\end{proof}

\begin{theorem}\label{thm:main}
There exists a transcendental entire function $f$ such that, for every
nonempty open set $U\subset\C$, there is an integer $N(U)$ for which
\[
 n\geq N(U)\quad\Longrightarrow\quad f^{(n)}\text{ has a zero in }U.
\]
The same function can be chosen to have order exactly two and to satisfy
\begin{equation}\label{eq:growth}
 |f(z)|\leq\sqrt2\exp(|z|^2),\qquad z\in\C.
\end{equation}
Consequently, the zero sets indexed by every strictly increasing sequence of
derivative orders have dense union.
\end{theorem}

The proof below shows that the selected function has order exactly two, while
\eqref{eq:growth} gives type at most one. The Boas--Reddy consequence, with
its source quantifier made explicit, is recorded in
Section~\ref{sec:boas-reddy}.

MacLane \cite{MacLane1952} constructed an entire function whose successive
derivatives are dense in the space of entire functions, and Barth and
Schneider \cite{BarthSchneider1972} prescribed discrete zero sets along one
constructed sequence of derivative orders. Neither result yields the
cofinite condition in Lemma~\ref{lem:cofinite}. Boas and Reddy's expanded
article \cite{BoasReddyJMAA1973} contained a further positive Theorem~2, but
their erratum withdrew its construction and left that conclusion open
\cite{BoasReddy1975}.

On 25 April 2026, Almeida first posted a detailed proposed solution and later
uploaded a manuscript using the planar Gaussian entire function
\cite{AlmeidaForum2026,Almeida2026}. Later that day Chojecki posted a second
Gaussian-route note; on 1 June he clarified that it was intended as a full
solution and reiterated that Almeida had posted first
\cite{ChojeckiForum2026}. These posts precede the present manuscript. The
author learned of them after completing and circulating this proof.

Almeida's manuscript and the present work share the cofinite reformulation,
the Fock normalization $1/\sqrt{k!}$ and a final Borel--Cantelli argument.
Almeida uses Gaussian coefficients, the Edelman--Kostlan expected zero
measure, Laguerre asymptotics and an Offord-type estimate. Chojecki's forum
description likewise invokes the cofinite reformulation and the
Edelman--Kostlan--Offord Gaussian route. The contribution here is an
alternative bounded-coefficient proof with a different local hole-probability
mechanism and a deterministic finite-type growth bound, the latter yielding
the Boas--Reddy counterexample.

Here is the mechanism. At each point, leave one coefficient near the saddle
index unconditioned; planar area then gives an exponential lower-tail bound
for the derivative. If $F_n$ has no zero in a disk, then $\log|F_n|$ is
harmonic there, so its circular mean equals its value at the centre. The
deterministic radial potential, however, has a strictly larger circular mean
than its central value. The difference must therefore appear in the
logarithmic deficit. A deficit of order $\sqrt n$ is exponentially unlikely,
which gives summable fixed-disk hole probabilities. Borel--Cantelli
\cite[Section~2.3]{Durrett} and a countable disk basis then produce one sample.
For background on random series, see Kahane \cite{Kahane}. As accessed 16
August 2026, the online record still displays \textsc{open} while recording
claimed solutions and warning that relevant literature may be missing
\cite{Bloom906,AlmeidaForum2026,ChojeckiForum2026}.

\section{A bounded random Fock series}\label{sec:fock}

Let $\overline\D$ denote the closed unit disk, equipped with normalised planar
Lebesgue measure $m_\D$. Let $\Omega=\overline\D^{\,\N}$, equip $\Omega$ with
the product topology and its Borel $\sigma$-algebra $\mathcal F$, and let
$\Pp=m_\D^{\otimes\N}$. A countable product of compact metric spaces is compact
and metrizable; in particular, the product $\sigma$-algebra here is precisely
the Borel $\sigma$-algebra $\mathcal F$. For
$\omega=(\omega_k)_{k\geq0}\in\Omega$ set
$\xi_k(\omega)=\omega_k$. The coordinate variables $\xi_0,\xi_1,\ldots$ are
then independent and uniformly distributed on $\overline\D$. Define
\begin{equation}\label{eq:fock}
 f(z)=\sum_{k=0}^{\infty}\xi_k\frac{z^k}{\sqrt{k!}}.
\end{equation}
For every $M\geq0$, the ratio test shows that
$\sum_{k\geq0}M^k/\sqrt{k!}$ converges. The Weierstrass M-test therefore gives
normal convergence on compact subsets of $\C$, uniformly in $\omega$. Hence
every sample defines an entire function, term-by-term differentiation is
legitimate, and for $n\geq0$,
\begin{equation}\label{eq:derivative}
 F_n(z):=f^{(n)}(z)
 =\sum_{m=0}^{\infty}\xi_{n+m}\frac{\sqrt{(n+m)!}}{m!}\,z^m.
\end{equation}
The partial sums of \eqref{eq:derivative} are continuous on
$\Omega\times\C$. On $\Omega\times\{|z|\leq M\}$, the absolute-value
majorant has successive-term ratio
$M\sqrt{n+m+1}/(m+1)\to0$; the ratio test and the Weierstrass M-test give
uniform convergence. Hence $(\omega,z)\mapsto F_n(\omega,z)$ is jointly
continuous; in particular $\omega\mapsto F_n(\omega,z)$ is
$\mathcal F$-measurable for each fixed $z$. The law $m_\D$ has no atoms, so
$\Pp\{\xi_k=0\}=0$ for every $k$,
and therefore
\begin{equation}\label{eq:nonzero}
 \Omega_0=\{\omega\in\Omega:\xi_k(\omega)\neq0\text{ for every }k\geq0\}
\end{equation}
has $\Pp(\Omega_0)=1$. On $\Omega_0$ all Taylor coefficients in
\eqref{eq:fock} are nonzero, so $f$ is not a polynomial.

The bounded coefficients also give the growth bound. By Cauchy--Schwarz with
weights $2^{k/2}$,
\begin{align*}
 |f(z)|
 &\leq\sum_{k\geq0}\frac{|z|^k}{\sqrt{k!}}\\
 &\leq\left(\sum_{k\geq0}\frac{2^k|z|^{2k}}{k!}\right)^{1/2}
      \left(\sum_{k\geq0}2^{-k}\right)^{1/2}
  =\sqrt2\exp(|z|^2).
\end{align*}
Thus \eqref{eq:growth} holds for every $\omega\in\Omega$. Put
\[
 \Omega_2=\{\omega:|\xi_k(\omega)|\geq\tfrac12
 \text{ for infinitely many }k\}.
\]
Independence, $\Pp\{|\xi_k|<1/2\}=1/4$ and the second Borel--Cantelli
lemma \cite[Section~2.3]{Durrett} give $\Pp(\Omega_2)=1$. For
$\omega\in\Omega_0\cap\Omega_2$, the coefficient formula for order gives
\[
 \rho(f)=\limsup_{k\to\infty}
 \frac{k\log k}{\frac12\log(k!)+\log(1/|\xi_k|)}.
\]
The growth bound gives $\rho(f)\leq2$. Along the subsequence on which
$|\xi_k|\geq1/2$, the displayed quotient tends to two because
$\log(k!)\sim k\log k$. Hence every sample in $\Omega_0\cap\Omega_2$ has
order exactly two and type at most one.

For $r\geq0$ write
\[
 a_{n,m}(r)=\frac{\sqrt{(n+m)!}}{m!}\,r^m,
 \qquad
 S_n(r)=\sum_{m=0}^{\infty}a_{n,m}(r),
\]
and set
\begin{equation}\label{eq:potential}
 p_n(z)=\tfrac12\log(n!)+|z|\sqrt n,
 \qquad
 L_n=2+\log(n+1).
\end{equation}
We shall use the elementary estimate
\[
 \log(m!)\leq m\log m-m+\log(m+1)+1,\qquad m\geq1.
\]
It follows from $\sum_{j=1}^m\log j\leq\int_1^{m+1}\log t\,dt$ and
$m\log(1+1/m)\leq1$. Its loss relative to Stirling's estimate is logarithmic,
and only $L_n=o(\sqrt n)$ is used below.

\begin{lemma}[Saddle estimates]\label{lem:saddle}
Let $n\geq1$.
\begin{enumerate}[\rm(1)]
\item If $r>0$, $1\leq r\sqrt n$ and $r\leq\sqrt n$, then, for
$m=\lfloor r\sqrt n\rfloor$,
\[
 \log a_{n,m}(r)\geq\tfrac12\log(n!)+r\sqrt n-2-\log(n+1);
\]
by \eqref{eq:potential} the right-hand side equals $p_n(z)-L_n$ for every
$z$ with $|z|=r$.
\item For every $r\geq0$,
\[
 \log S_n(r)\leq\tfrac12\log(n!)+r\sqrt n+\frac{r+r^2}{2}.
\]
Consequently, if $|z|\leq R_0$, then
\begin{equation}\label{eq:upper}
 \log|F_n(z)|-p_n(z)\leq\frac{R_0+R_0^2}{2}
\end{equation}
for every $\omega\in\Omega$, with the inequality read trivially when
$F_n(z)=0$, using the extended-real convention $\log0=-\infty$.
\end{enumerate}
\end{lemma}

\begin{proof}
Factor $a_{n,m}(r)=\sqrt{n!}\,b_{n,m}(r)$, where
\[
 b_{n,m}(r)=\frac{\sqrt{(n+1)(n+2)\cdots(n+m)}}{m!}\,r^m,
\]
which is legitimate because $(n+m)!=n!\,(n+1)(n+2)\cdots(n+m)$.

For item (1) put $x=r\sqrt n$. Each of the $m$ factors under the square root is at
least $n$, so $b_{n,m}(r)\geq n^{m/2}r^m/m!=x^m/m!$. For $m=\lfloor x\rfloor$
the hypotheses give $1\leq m\leq x\leq n$, since $x\geq1$ and
$x=r\sqrt n\leq\sqrt n\sqrt n=n$. By the preceding estimate,
\[
 \log\frac{x^m}{m!}
 \geq m\log x-m\log m+m-\log(m+1)-1
 =m\log\frac xm+m-\log(m+1)-1.
\]
Here $m\log(x/m)\geq0$ because $m\leq x$; also $m>x-1$ and
$\log(m+1)\leq\log(n+1)$ because $m\leq n$. Hence
\[
 \log\frac{x^m}{m!}\geq(x-1)-\log(n+1)-1=x-2-\log(n+1),
\]
and adding $\tfrac12\log(n!)$ to both sides gives item~(1); the identification of
the right-hand side with $p_n(z)-L_n$ is the definition
\eqref{eq:potential}.

For item (2), note that
\[
 b_{n,m}(r)=\sqrt{\frac{\binom{n+m}{n}}{m!}}\,r^m .
\]
If $r>0$, put $t=r/(r+\sqrt n)\in(0,1)$. Then
\[
 \frac{\binom{n+m}{n}r^{2m}}{m!}
 =\left(\binom{n+m}{n}t^m\right)\frac{(r^2/t)^m}{m!}.
\]
Cauchy--Schwarz, the binomial series
$\sum_{m\geq0}\binom{n+m}{n}t^m=(1-t)^{-(n+1)}$ and the exponential series
give
\begin{align*}
 \sum_{m\geq0}b_{n,m}(r)
 &\leq\left(\sum_{m\geq0}\binom{n+m}{n}t^m\right)^{1/2}
      \left(\sum_{m\geq0}\frac{(r^2/t)^m}{m!}\right)^{1/2}\\
 &=(1-t)^{-(n+1)/2}\exp\left(\frac{r^2}{2t}\right).
\end{align*}
Put $u=r/\sqrt n$. Since $\log(1+u)\leq u$ for $u\geq0$ and
$\sqrt n\geq1$,
\begin{align*}
 \frac{n+1}{2}\log(1+u)+\frac{r(r+\sqrt n)}{2}
 &\leq\frac{(n+1)r}{2\sqrt n}+\frac{r^2+r\sqrt n}{2}\\
 &=r\sqrt n+\frac{r}{2\sqrt n}+\frac{r^2}{2}\\
 &\leq r\sqrt n+\frac{r+r^2}{2}.
\end{align*}
Adding $\tfrac12\log(n!)$ gives item~(2) for $r>0$; the case $r=0$ is immediate
because $S_n(0)=\sqrt{n!}$. Finally $|\xi_{n+m}|\leq1$, so
\eqref{eq:derivative} and the triangle inequality give
$|F_n(z)|\leq S_n(|z|)$. Applying the bound just proved with $r=|z|$, and
using that $r\mapsto(r+r^2)/2$ is nondecreasing, gives \eqref{eq:upper}
whenever $|z|\leq R$.
\end{proof}

Part~(1) supplies the large coefficient used in the pointwise lower-tail
estimate below. Inequality~\eqref{eq:upper}, from part~(2), will later control
the value of $F_n$ at the centre of a putative zero-free disk.

The second ingredient is anti-concentration. If $\xi$ is uniform on
$\overline\D$ then, for $a\neq0$, $w\in\C$ and $\varepsilon\geq0$,
\begin{equation}\label{eq:smallball}
 \Pp\{|a\xi+w|<\varepsilon\}\leq\left(\frac{\varepsilon}{|a|}\right)^2,
\end{equation}
because the set $\{\zeta:|a\zeta+w|<\varepsilon\}$ is the disk of centre
$-w/a$ and radius $\varepsilon/|a|$, whose normalised area is at most
$(\varepsilon/|a|)^2$.

\begin{lemma}[Pointwise logarithmic tail]\label{lem:pointwise}
Let $n\geq1$ and let $z\in\C$ satisfy $1\leq|z|\sqrt n$ and $|z|\leq\sqrt n$.
Then, for every $s\geq0$,
\begin{equation}\label{eq:pointwise-tail}
 \Pp\bigl\{p_n(z)-\log|F_n(z)|>s+L_n\bigr\}\leq e^{-2s}.
\end{equation}
\end{lemma}

\begin{proof}
Put $m=\lfloor|z|\sqrt n\rfloor$ and $k=n+m$, and split \eqref{eq:derivative}
by isolating the coefficient $\xi_k$:
\[
 F_n(\omega,z)=a\,\xi_k(\omega)+w(\omega),
\]
where
\[
 a=\frac{\sqrt{(n+m)!}}{m!}\,z^m
\]
is deterministic and
\[
 w(\omega)=\sum_{j\neq m}\xi_{n+j}(\omega)
 \frac{\sqrt{(n+j)!}}{j!}\,z^j
\]
is the remainder.
We first check that $w$ is finite and measurable. The series defining $w$
converges absolutely for every $\omega$, since $|\xi_{n+j}|\leq1$ and
$\sum_{j\geq0}a_{n,j}(|z|)=S_n(|z|)<\infty$ by
item~(2) of Lemma~\ref{lem:saddle}; so $w$ is finite everywhere, not merely almost
surely. It is also a measurable function of $\omega$, being a uniform
limit of continuous functions on $\Omega$. The independence needed below
comes from the fact that $w$ does not depend on the coordinate $\omega_k$:
it is measurable with respect to
$\mathcal G=\sigma(\xi_j:j\neq k)$, since the index $j=m$ is omitted from the
sum. Note also that $|z|\geq1/\sqrt n>0$, so $a\neq0$ and
$|a|=a_{n,m}(|z|)$.

Since $|z|=:r$ satisfies the hypotheses of item~(1) of
Lemma~\ref{lem:saddle}, that
lemma gives $\log|a|\geq p_n(z)-L_n$. Hence, for $s\geq0$,
\[
 \bigl\{p_n(z)-\log|F_n(z)|>s+L_n\bigr\}
 =\bigl\{|F_n(z)|<e^{\,p_n(z)-L_n-s}\bigr\}
 \subseteq\bigl\{|a\xi_k+w|<|a|e^{-s}\bigr\}.
\]
Now write $\Omega=\overline\D\times\Omega'$, where the first factor carries
the coordinate $\omega_k$ and $\Omega'=\overline\D^{\,\N\setminus\{k\}}$
carries the remaining coordinates, and correspondingly
$\Pp=m_\D\otimes\Pp'$. Since the defining series for $w$ omits the coordinate
$\omega_k$, it is a measurable function of $\omega'\in\Omega'$ alone; this is
equivalent to the $\mathcal G$-measurability established above. Tonelli's
theorem therefore gives
\[
 \Pp\bigl\{|a\xi_k+w|<|a|e^{-s}\bigr\}
 =\int_{\Omega'}m_\D\bigl\{\zeta\in\overline\D:
   |a\zeta+w(\omega')|<|a|e^{-s}\bigr\}\,d\Pp'(\omega').
\]
For each fixed $\omega'$ the inner measure is at most $(e^{-s})^2=e^{-2s}$ by
\eqref{eq:smallball}, with $\varepsilon=|a|e^{-s}$; the bound is uniform in
$\omega'$, so integrating over $\Omega'$ preserves it. This proves
\eqref{eq:pointwise-tail}.
\end{proof}

The pointwise estimate is used through the following circle-average bound.

\begin{lemma}[Circle-average upper tail]\label{lem:circle-tail}
Let $Y\geq0$ be jointly measurable on $\Omega\times[0,2\pi]$ and suppose that
\[
 \Pp\{Y(\cdot,\theta)>s\}\leq e^{-2s},\qquad s\geq0,
\]
for every $\theta\in[0,2\pi]$. Thus, for each fixed $\theta$, the map
$\omega\mapsto Y(\omega,\theta)$ is a nonnegative random variable; $\omega$ is
the sample point and $\theta$ is the angular parameter. Put
$\Av Y(\omega)=(2\pi)^{-1}\int_0^{2\pi}Y(\omega,\theta)\,d\theta$. Then
\[
 \Pp\{\Av Y\geq t\}\leq3e^{-2t},\qquad t>0.
\]
\end{lemma}

\begin{proof}
The tail hypothesis and the layer-cake formula give
\[
 \E Y(\cdot,\theta)
 =\int_0^\infty\Pp\{Y(\cdot,\theta)>s\}\,ds\leq\tfrac12
\]
for every $\theta$. Hence $\E(\Av Y)\leq\tfrac12$ by Tonelli's theorem, and
$\Av Y$ is integrable. Put $E=\{\Av Y\geq t\}$ and $q=\Pp(E)$. Then
$t\ind_E\leq\Av Y\ind_E$, and Tonelli's theorem gives
\[
 tq\leq\E\bigl((\Av Y)\ind_E\bigr)
 =\frac1{2\pi}\int_0^{2\pi}\E\bigl(Y(\cdot,\theta)\ind_E\bigr)\,d\theta.
\]
For every $\theta$, the layer-cake formula and
\[
 \Pp(\{Y(\cdot,\theta)>s\}\cap E)
 \leq\min\{q,e^{-2s}\}
\]
yield
\begin{align*}
 \E\bigl(Y(\cdot,\theta)\ind_E\bigr)
 &\leq\int_0^\infty\min\{q,e^{-2s}\}\,ds\\
 &=\int_0^{\frac12\log(1/q)}q\,ds
   +\int_{\frac12\log(1/q)}^\infty e^{-2s}\,ds\\
 &=\frac q2\log\frac eq,
\end{align*}
where the computation assumes $0<q\leq1$. If $q>0$ we may cancel $q$ and
obtain $t\leq\tfrac12\log(e/q)$, so $q\leq e^{1-2t}\leq3e^{-2t}$. The case
$q=0$ is immediate.
\end{proof}

\section{Zero-free disks and proof of the theorem}\label{sec:holes}

We first isolate the geometric gain that drives the estimate. For $c\in\C$
and $r>0$ set
\[
 \gamma(c,r)=\frac1{2\pi}\int_0^{2\pi}|c+re^{i\theta}|\,d\theta-|c|.
\]

\begin{lemma}[Strict circular-mean gap]\label{lem:gap}
For all $c\in\C$ and $r>0$ we have $\gamma(c,r)>0$.
\end{lemma}

\begin{proof}
Set
\[
 \varphi(\theta)=|c+re^{i\theta}|+|c-re^{i\theta}|-2|c|.
\]
The triangle inequality gives $\varphi\geq0$, and $\varphi$ is continuous. If
$c=0$, then $\varphi\equiv2r>0$. If $c\neq0$, choose $\theta_0$ with
$re^{i\theta_0}=irc/|c|$. The Pythagorean identity gives
\[
 \varphi(\theta_0)=2\sqrt{|c|^2+r^2}-2|c|>0.
\]
Thus continuity implies
$\int_0^{2\pi}\varphi(\theta)\,d\theta>0$. The substitution
$\theta\mapsto\theta+\pi$ gives
\[
 2\gamma(c,r)
 =\frac1{2\pi}\int_0^{2\pi}\varphi(\theta)\,d\theta>0,
\]
as required.
\end{proof}

Next we record the mean value step in the exact form needed, with the strict
radius margin made explicit.

\begin{lemma}[Mean value property on a zero-free disk]\label{lem:mvp}
Let $c\in\C$ and $\rho>0$, let $G$ be holomorphic and zero-free on
$B(c,\rho)$, and let $0<r<\rho$. Then
\[
 \frac1{2\pi}\int_0^{2\pi}\log|G(c+re^{i\theta})|\,d\theta=\log|G(c)|.
\]
\end{lemma}

\begin{proof}
Since $B(c,\rho)$ is simply connected and $G$ is zero-free there, $G$ has a
holomorphic logarithm $g$ on $B(c,\rho)$ \cite[Chapter~IV]{Conway}. Thus
$\log|G|=\Real g$ is harmonic. The strict inequality $r<\rho$ gives
$\overline{B(c,r)}\subset B(c,\rho)$, so the harmonic mean-value property on
$B(c,r)$ gives the asserted identity.
\end{proof}

Now fix $c\in\Q+i\Q$ with $c\neq0$ and a rational $R$ with $0<R<|c|$, and put
\begin{equation}\label{eq:margin}
 r=\frac R2<R,
\end{equation}
so that $\overline{B(c,r)}\subset B(c,R)$ with a strict radius margin. On the
circle $|z-c|=r$ we have
\[
 0<\delta:=|c|-r\leq|z|\leq|c|+r=:M,
\]
the left inequality because $r<R<|c|$. Consequently there is an integer
$n_0=n_0(c,R)\geq1$ such that
\begin{equation}\label{eq:largen}
 1\leq\delta\sqrt n\quad\text{and}\quad M\leq\sqrt n
 \qquad\text{for all }n\geq n_0,
\end{equation}
and then the hypotheses of Lemma~\ref{lem:pointwise} hold at every point of
the circle $|z-c|=r$, with the same $n_0$ for every $\theta\in[0,2\pi]$.

On a zero-free disk, the circular average of $p_n$ gains
$\gamma(c,r)\sqrt n$ over its centre, whereas the circular average of
$\log|F_n|$ equals its central value. We record the resulting pointwise
shortfall in a real-valued measurable form, without changing it on the
zero-free event. Define
\[
 \ell(w)=\begin{cases}\log|w|,&w\neq0,\\ 0,&w=0,\end{cases}
 \qquad
 D_{n,z}=\max\{p_n(z)-\ell(F_n(z))-L_n,\,0\},
\]
and, for $n\geq n_0$,
\[
 \Av D_n=\frac1{2\pi}\int_0^{2\pi}D_{n,c+re^{i\theta}}\,d\theta .
\]
The function $\ell$ is Borel measurable on $\C$, and $(\omega,z)\mapsto
F_n(\omega,z)$ is jointly continuous, so
$(\omega,\theta)\mapsto D_{n,c+re^{i\theta}}(\omega)$ is jointly measurable
on $\Omega\times[0,2\pi]$ and $\Av D_n$ is a measurable function of
$\omega$.

Because $\ell(w)\geq\log|w|$ for every $w\in\C$ in the
extended-real order, with equality unless $w=0$, we have, for $s\geq0$,
\[
 \{D_{n,z}>s\}\subseteq
 \{p_n(z)-\log|F_n(z)|>s+L_n\}.
\]
Combining this with Lemma~\ref{lem:pointwise} gives the tail hypothesis of
Lemma~\ref{lem:circle-tail} explicitly: for $n\geq n_0$, every $z$ on the
circle $|z-c|=r$ and every $s\geq0$,
\begin{equation}\label{eq:bridge}
 \Pp\{D_{n,z}>s\}
 \leq\Pp\{p_n(z)-\log|F_n(z)|>s+L_n\}
 \leq e^{-2s}.
\end{equation}
So Lemma~\ref{lem:circle-tail} applies to $Y(\omega,\theta)
=D_{n,c+re^{i\theta}}(\omega)$, with $\Av Y=\Av D_n$.

Let
\[
 H_n(c,R)=\{\omega\in\Omega:F_n(\omega,\cdot)\text{ has no zero in }B(c,R)\}.
\]

\begin{lemma}[Measurability of the hole event]\label{lem:measurable}
$H_n(c,R)\in\mathcal F$.
\end{lemma}

\begin{proof}
For rational $0<\rho<R$, let
\[
 Q_\rho=(\Q+i\Q)\cap B(c,\rho),\qquad
 M_\rho(\omega)=\inf_{z\in Q_\rho}|F_n(\omega,z)|.
\]
The set $Q_\rho$ is countable and dense in $\overline{B(c,\rho)}$, including
by interior approximation at boundary points. Hence continuity in $z$ gives
\[
 M_\rho(\omega)=\min_{|z-c|\leq\rho}|F_n(\omega,z)|,
\]
while measurability of each evaluation makes $M_\rho$ measurable. Therefore
\[
 H_n(c,R)=\bigcap_{\substack{\rho\in\Q\\0<\rho<R}}\{M_\rho>0\}
 \in\mathcal F.
\]
Indeed, zero-freeness on $B(c,R)$ gives positivity on every compact subdisk,
and the converse follows by choosing rational $\rho$ with
$|z-c|\leq\rho<R$ for each $z\in B(c,R)$.
\end{proof}

We can now estimate $\Pp(H_n(c,R))$. Fix $n\geq n_0$ and work on the event
$H_n(c,R)$, on which $F_n$ is holomorphic and zero-free on $B(c,R)$ by the
definition of that event. Since $r=R/2<R$ by \eqref{eq:margin},
Lemma~\ref{lem:mvp} applies with $G=F_n$ and $\rho=R$ and gives
\begin{equation}\label{eq:jensen}
 \frac1{2\pi}\int_0^{2\pi}\log|F_n(c+re^{i\theta})|\,d\theta
 =\log|F_n(c)| .
\end{equation}
Inequality \eqref{eq:upper}, applied with $R_0=|c|$ and $z=c$, gives
\begin{equation}\label{eq:centre}
 \log|F_n(c)|\leq p_n(c)+C(c),\qquad C(c)=\frac{|c|+|c|^2}{2}.
\end{equation}
On $H_n(c,R)$ we have $F_n(z)\neq0$ for $z\in\overline{B(c,r)}$, so
$\ell(F_n(z))=\log|F_n(z)|$ there and therefore
$D_{n,z}\geq p_n(z)-\log|F_n(z)|-L_n$ on the circle $|z-c|=r$. Averaging over
that circle and using
\[
 \frac1{2\pi}\int_0^{2\pi}p_n(c+re^{i\theta})\,d\theta
 =\tfrac12\log(n!)+\sqrt n\cdot\frac1{2\pi}\int_0^{2\pi}|c+re^{i\theta}|\,d\theta
 =p_n(c)+\gamma(c,r)\sqrt n,
\]
together with \eqref{eq:jensen} and \eqref{eq:centre}, we obtain on
$H_n(c,R)$ that
\[
 \Av D_n
 \geq p_n(c)+\gamma(c,r)\sqrt n-\log|F_n(c)|-L_n
 \geq\gamma(c,r)\sqrt n-C(c)-L_n=:T_n,
\]
that is,
\begin{equation}\label{eq:inclusion}
 H_n(c,R)\subseteq\{\Av D_n\geq T_n\},
 \qquad
 T_n=\gamma(c,r)\sqrt n-C(c)-2-\log(n+1).
\end{equation}
By Lemma~\ref{lem:gap} we have $\gamma(c,r)>0$, and
$C(c)+2+\log(n+1)=o(\sqrt n)$, so there is $n_1=n_1(c,R)\geq n_0$ with
\begin{equation}\label{eq:Tn}
 T_n\geq\tfrac12\gamma(c,r)\sqrt n>0,
 \qquad n\geq n_1 .
\end{equation}
For $n\geq n_1$, Lemma~\ref{lem:circle-tail} together with
\eqref{eq:bridge}, \eqref{eq:inclusion} and \eqref{eq:Tn} gives
\begin{equation}\label{eq:hole}
 \Pp\bigl(H_n(c,R)\bigr)\leq\Pp\{\Av D_n\geq T_n\}
 \leq3e^{-2T_n}
 \leq3\exp\bigl(-\gamma(c,r)\sqrt n\bigr).
\end{equation}
The right-hand side of \eqref{eq:hole} is summable, since
$\exp(-\gamma(c,r)\sqrt n)\leq n^{-2}$ for all large $n$, and the finitely
many terms with $n<n_1$ are each at most $1$. Hence
\begin{equation}\label{eq:summable-holes}
 \sum_{n=0}^{\infty}\Pp\bigl(H_n(c,R)\bigr)<\infty .
\end{equation}
All the events involved are measurable by Lemma~\ref{lem:measurable}, so the
first Borel--Cantelli lemma \cite[Section~2.3]{Durrett} applies and shows
that
\begin{equation}\label{eq:bc}
 \Pp\Bigl(\limsup_{n\to\infty}H_n(c,R)\Bigr)=0;
\end{equation}
that is, almost surely only finitely many $F_n$ are zero-free in $B(c,R)$.

\begin{proof}[Proof of Theorem~\ref{thm:main}]
Let $\mathcal B$ be the family of disks $B(c,R)$ with $c\in\Q+i\Q$,
$c\neq0$, $R\in\Q_{>0}$ and $R<|c|$; it is countable. We check that every
nonempty open set $U\subset\C$ contains a member of $\mathcal B$. Since $U$
is open and nonempty it contains a point $z_0\neq0$, and there is
$\varepsilon>0$ with $B(z_0,\varepsilon)\subset U$ and
$\varepsilon<|z_0|/2$. Choose $c\in\Q+i\Q$ with $|c-z_0|<\varepsilon/4$;
then $|c|\geq|z_0|-\varepsilon/4>0$, so $c\neq0$. Choose a rational $R$ with
$0<R<\min\{\varepsilon/2,|c|\}$. Then $R<|c|$, and every $z$ with $|z-c|<R$
satisfies $|z-z_0|<\varepsilon/4+\varepsilon/2<\varepsilon$, so
$B(c,R)\subset B(z_0,\varepsilon)\subset U$ and $B(c,R)\in\mathcal B$.

By \eqref{eq:bc}, for each $B=B(c,R)\in\mathcal B$ there is a
$\Pp$-null set $\mathcal N_B$ outside of which only finitely many $F_n$ are
zero-free in $B$. The union $\mathcal N=\bigcup_{B\in\mathcal B}\mathcal N_B$
is a countable union of null sets and is therefore null. With $\Omega_0$ as
in \eqref{eq:nonzero} and $\Omega_2$ as above, we have
$\Pp((\Omega_0\cap\Omega_2)\setminus\mathcal N)=1$; in particular this event
is nonempty. Fix $\omega\in(\Omega_0\cap\Omega_2)\setminus\mathcal N$ and let
$f$ be the entire function \eqref{eq:fock} determined by it.

Then $f$ is transcendental, because $\omega\in\Omega_0$ makes every Taylor
coefficient of $f$ nonzero; moreover $\omega\in\Omega_2$ makes its order
exactly two, and $f$ satisfies \eqref{eq:growth}. Let $U\subset\C$ be
nonempty and open and choose
$B\in\mathcal B$ with $B\subset U$. Since $\omega\notin\mathcal N_B$, only
finitely many $F_n=f^{(n)}$ are zero-free in $B$. Let $N(U)$ be one more than
the largest such index, or let $N(U)=0$ if there is no such index. Then every
$n\geq N(U)$ gives a zero of $f^{(n)}$ in $B$, hence in $U$. This is the first
assertion of Theorem~\ref{thm:main}, and the last assertion follows from it by
Lemma~\ref{lem:cofinite}.
\end{proof}

\section{A consequence for a theorem of Boas and Reddy}
\label{sec:boas-reddy}

Under their standing assumption that $f$ is transcendental, Boas and Reddy
assert that if $f$ is of order at most two and finite type, then there is an
arbitrarily large disk on which $f^{(k)}$ is zero-free for infinitely many
values of $k$ \cite[Theorem~1]{BoasReddyJMAA1973}; see also their announcement
\cite[Theorem~1]{BoasReddy1973}. They clarify that one fixed disk serves
infinitely many derivatives \cite[p.~65]{BoasReddy1973}. Their expanded article also
stated a distinct Theorem~2 \cite{BoasReddyJMAA1973}.
Their published erratum addresses Theorem~2 \cite{BoasReddy1975}.
It withdraws the supporting construction and says that conclusion remains
open.

The proof of Theorem~1 in the expanded article has a quantifier gap before its
rescaling step. Under the negation used there, fix $c_1$ larger than the radius
$c$ introduced in the proof. For each fixed centre, the disk of radius $c_1$
contains a zero of $f^{(k)}$ for all but finitely many $k$. Hence, for each
fixed outer radius $R$, the finitely many disks in a packing admit a common
exceptional threshold $K_R$. The proof then sets
$R=(k\Delta)^{1/2}$ along a subsequence supplied by Lemma~6, but gives no
control of $K_R$ and therefore no reason that
$k\geq K_{(k\Delta)^{1/2}}$. The phrase ``uniformly in $r$'' in Lemma~6
concerns the zero-count estimate $N_k(r)$, not these centre-dependent
exceptional thresholds \cite[p.~469]{BoasReddyJMAA1973}.

Even if that step is granted, division of equation~(3.1) there by $k\Delta$
and passage to large $k$ would require
\[
 \tau+\varepsilon
 <e^{-3}\left(\frac{s}{c_1^2}-\frac1\Delta\right),
\]
where $s$ is the packing constant and $\Delta$ is chosen after $c_1$.
The subsequent rescaling does not remove this dependence: for
$g(z)=f(az)$, the order-two type becomes $|a|^2\tau$ while the corresponding
radius becomes $c_1/|a|$, so $\tau c_1^2$ is unchanged. Thus the printed
proof does not establish Theorem~1 for arbitrary finite type at order two.
The moving-packing gap also remains below order two even when the order-two
type is zero. This diagnosis concerns the printed proof only; it does not rule
out another proof of the claimed conclusion
\cite[pp.~469, 472]{BoasReddyJMAA1973}.

\begin{corollary}
The function in Theorem~\ref{thm:main} is a counterexample to the assertions
of \cite[Theorem~1]{BoasReddyJMAA1973} and
\cite[Theorem~1]{BoasReddy1973} as printed.
\end{corollary}

\begin{proof}
Let $f$ be the function supplied by Theorem~\ref{thm:main}. The estimate
\eqref{eq:growth} places $f$ in the growth class assumed by Boas and Reddy.
For every fixed nonempty disk $D$, Theorem~\ref{thm:main} supplies an integer
$N(D)$ such that $f^{(n)}$ has a zero in $D$ whenever $n\geq N(D)$.
Consequently only the finitely many derivatives with $n<N(D)$ can be
zero-free in $D$. Thus no disk, of any radius, is zero-free for infinitely
many successive derivatives of $f$, contrary to the printed conclusion.
\end{proof}

\section*{Acknowledgements}
The author thanks Xi Chen for detailed suggestions that improved the
introduction, theorem formulation, order calculation and circle-average
estimate. He also thanks Harold P. Boas for detailed comments on an earlier
version.

\section*{Declarations}
\noindent\textit{Use of AI.} In July--August 2026, the author used OpenAI
GPT-5.6 Sol, Anthropic Claude Fable~5 and Anthropic Claude Opus~5 to assist
with deriving the small-ball estimate, editing prose, retrieving sources,
revising the manuscript and auditing arguments. The author checked every AI-assisted
mathematical and bibliographic claim and takes sole responsibility for the
mathematical content and exposition of this article.

\noindent\textit{Competing interests.} The author declares no competing
interests. \textit{Funding.} The author received no specific funding for this
work. \textit{Data availability.} No data were generated or analysed in this
theoretical study.

\end{document}